\documentclass[letterpaper, 10 pt, conference]{ieeeconf}  

\IEEEoverridecommandlockouts                              
\usepackage[OT1]{fontenc} 
\usepackage[english]{babel} 
\usepackage[driverfallback=dvipdfmx,hidelinks]{hyperref}
\usepackage{amsmath,amssymb,mathrsfs}
\usepackage{cleveref}
\usepackage{bm}
\usepackage{array}
\usepackage{multirow}
\usepackage[hang]{footmisc}

\usepackage{booktabs}
\usepackage{algorithm}
\usepackage{algpseudocodex}
\usepackage{comment}
\usepackage{subfigure}
\usepackage{tabularx}
\usepackage{bibentry}
\usepackage[noadjust]{cite}
\usepackage{bbm}
\usepackage{menukeys}
\usepackage[short]{optidef}

\usepackage{amsthm}
\newtheorem{problem}{Problem}
\newtheorem{theorem}{Theorem}

\newtheorem{lemma}{Lemma}
\newtheorem{remark}{Remark}

\newcommand{\norm}[1]{\lVert#1\rVert}
\newcommand{\abs}[1]{\lvert#1\rvert}
\newcommand{\R}{\mathbb{R}}
\newcommand{\N}{\mathbb{N}}

\newcommand{\diff}{\mathrm{d}}
\newcommand{\tr}{\textcolor{red}}
\newcommand{\tb}{\textcolor{blue}}

\crefname{align}{Eq.}{Eqs.}
\crefname{equation}{Eq.}{Eqs.}
\crefname{figure}{Fig.}{Figs.}
\crefname{table}{Table}{Tables}
\crefname{theorem}{Theorem}{Theorems}
\crefname{definition}{Definition}{Definitions}
\crefname{lemma}{Lemma}{Lemmas}
\crefname{remark}{Remark}{Remarsks}
\crefname{assumption}{Assumption}{Assumptions}
\crefname{proof}{Proof}{Proofs}
\crefname{algorithm}{Algorithm}{Algorithms}
\crefname{problem}{Problem}{Problems}
\crefname{proposition}{Proposition}{Propositions}
\crefname{corollary}{Corollary}{Corollaries}
\crefname{section}{Section}{Sections}

\usepackage[numbered,framed]{matlab-prettifier}
\usepackage{filecontents}

\graphicspath{ {./img/} }

\title{
Successive Convexification with Feasibility Guarantee via Augmented Lagrangian for Non-Convex Optimal Control Problems
}

\author{Kenshiro Oguri
\thanks{K.~Oguri is with the School of Aeronautics and Astronautics, Purdue University, IN 47907, USA ({\tt koguri@purdue.edu}).
}}

\def\shortOrFull{2} 

\usepackage{tikz}
\usepackage{textcomp}
\newcommand\copyrighttext{%
  \footnotesize \textcopyright 2023 IEEE. Personal use of this material is permitted.
  Permission from IEEE must be obtained for all other uses, in any current or future
  media, including reprinting/republishing this material for advertising or promotional
  purposes, creating new collective works, for resale or redistribution to servers or
  lists, or reuse of any copyrighted component of this work in other works.
  DOI: \underline{\href{https://ieeexplore.ieee.org/document/10383462}{10.1109/CDC49753.2023.10383462}}.
  Presented at 2023 IEEE CDC.
  }
\newcommand\copyrightnotice{%
\begin{tikzpicture}[remember picture,overlay]
\node[anchor=south,yshift=10pt] at (current page.south) {\fbox{\parbox{\dimexpr\textwidth-\fboxsep-\fboxrule\relax}{\copyrighttext}}};
\end{tikzpicture}%
}

\begin{document}

\maketitle
\thispagestyle{empty}
\pagestyle{empty}
\copyrightnotice

\begin{abstract}
This paper proposes an algorithm that solves non-convex optimal control problems with a theoretical guarantee for global convergence to a feasible local solution of the original problem.
The proposed algorithm extends the recently proposed successive convexification (\texttt{SCvx}) algorithm to
address its key limitation: lack of feasibility guarantee to the original non-convex problem.
The main idea of the proposed algorithm is to incorporate the \texttt{SCvx} iteration into an algorithmic framework based on the augmented Lagrangian method to enable the feasibility guarantee while retaining favorable properties of \texttt{SCvx}.
Unlike the original \texttt{SCvx}, our approach iterates on both of the optimization variables and the Lagrange multipliers, which facilitates the feasibility guarantee as well as efficient convergence, in a spirit similar to the alternating direction method of multipliers (ADMM).
Convergence analysis shows the proposed algorithm's strong global convergence to a \textit{feasible} local optimum of the original problem and its convergence rate.
These theoretical results are demonstrated via numerical examples with comparison against the original \texttt{SCvx} algorithm.
\end{abstract}



\section{Introduction}
\label{sec:intro}

This paper proposes a new algorithm for solving non-convex optimal control problems by extending the successive convexification (\texttt{SCvx}) algorithm \cite{Mao2019a,Mao2016}, a recent algorithm based on sequential convex programming (SCP), and presents a convergence analysis of the proposed algorithm.
Most of the real-world problems are non-convex, as seen in aerospace, robotics, and other engineering applications, due to their nonlinear dynamics and/or non-convex constraints.
While {lossless convexification} is available for a certain class of problems \cite{Acikmese2007,Akmee2011a,Harris2021}, many remain non-convex.
Among various options to tackle non-convex problems \cite{Nocedal2006}, 
SCP is gaining renewed interest as a powerful tool in light of the recent advance in convex programming \cite{Malyuta2022,Mao2019a,Mao2016,Bonalli2019b,Bonalli2022a,Boyd2004}.

Any SCP algorithms take the approach that repeats the \textit{convexify-and-solve} process to march toward a local solution of a non-convex problem.
This approach is common to many optimization algorithms, including difference of convex programming, which decomposes a problem into convex and concave parts and approximates the concave part \cite{Yuille2003}; and sequential quadratic programming \cite{Boggs1989}, which is adopted in many nonlinear programming software, such as SNOPT \cite{Gill2002}.

Among recent SCP algorithms, \texttt{SCvx} \cite{Mao2019a,Mao2016} (and its variation, \texttt{SCvx-fast} \cite{Mao2021}) and guaranteed sequential trajectory optimization (\texttt{GuSTO}) \cite{Bonalli2019b,Bonalli2022a} are arguably the most notable ones due to their rigorous theoretical underpinnings, with successful application to various problems \cite{Szmuk2020,Bonalli2019a}.
See \cite{Malyuta2022} for a comprehensive review. 
In particular, \cite{Mao2019a} (\texttt{SCvx}) and  \cite{Bonalli2022a} (\texttt{GuSTO}) analyze the performance of these algorithms in depth, and provide theoretical guarantees on their powerful capabilities, through different approaches: the Karush–Kuhn–Tucker (KKT) conditions for \texttt{SCvx} while Pontryagin's minimum principle for \texttt{GuSTO}.
On the other hand, like any algorithms, each algorithm has their own limitations, including:
the convergence guarantee to a KKT point of the penalty problem but not of the original problem (i.e., lack of feasibility guarantee) in \texttt{SCvx} \cite{Mao2019a};
the requirement on the dynamical systems to be control-affine in \texttt{GuSTO} \cite{Bonalli2022a}.
Note that \texttt{SCvx} achieves a solution to the original problem for a ``sufficiently large'' penalty weight \cite{Mao2019a}, which the user needs to find through trials and errors.


The objective of this paper is to fill the gap in those theoretical aspects of the existing SCP algorithms by building on the algorithmic foundation laid by \texttt{SCvx}.
This study proposes a new SCP algorithm that guarantees feasibility to the \textit{original} problem while retaining \texttt{SCvx}'s favorable properties, including the minimal requirements on the dynamical system.
The main idea is to integrate the \texttt{SCvx} iteration into an algorithmic framework of the {augmented Lagrangian (AL) method} \cite{Bertsekas1982}.
The AL method is a nonlinear programming technique proposed by Hestenes \cite{Hestenes1969} and Powell \cite{Powell1978} in 1960s to iteratively improve the multiplier estimate, addressing drawbacks of the quadratic penalty method \cite{Bertsekas1982}.

The proposed algorithm iterates on the original optimization variables \textit{and} the multipliers of the associated Lagrangian function in the primal-dual formalism.
This facilitates the feasibility guarantee as well as efficient convergence, in a spirit similar to the alternating direction method of multipliers (ADMM) \cite{Boyd2011}.
The proposed algorithm is named \texttt{SCvx*}, as it inherits key properties of \texttt{SCvx} and augments it by the feasibility guarantee, represented by ``\texttt{*}''.
A preliminary version of \texttt{SCvx*} has been successfully applied to space trajectory optimization under uncertainty \cite{Oguri2022f}, although the present paper is the first to provide comprehensive, rigorous convergence analysis of \texttt{SCvx*}.

The main contributions of \texttt{SCvx*} are threefold.
Under the Assumption 1 in \cref{sec:problemStatement}, the proposed algorithm
\begin{enumerate}
\item 
provides the convergence guarantee to a \textit{feasible} local optimum of the original problem, eliminating the need of trials and errors for tuning the penalty weight;
\item 
provides the strong global convergence to a single local solution of the original problem with minimal requirements on the problem structure; and
\item 
provides the linear/superlinear convergence rate of Lagrange multipliers with a slight algorithm modification.
\end{enumerate}


\section{Preliminary}
\label{sec:preliminary}

\subsection{Problem Statement}
\label{sec:problemStatement}

We consider solving discrete-time non-convex optimal control problems given by \cref{prob:OCP}, where
$x_s \in \R^{n_x} $ and $u_s \in\R^{n_u} $ represent the state and control at $s$-th time instance ($n_x,n_u\in\N$).
$g_{\mathrm{affine}}(\cdot) $ and $g_i(\cdot)$ represent the affine and non-affine equality constraint functions while $h_{\mathrm{cvx}}(\cdot) $ and $h_j(\cdot)$ the convex and non-convex inequality constraint functions.
$N, p, q\in\N $ are the number of discrete time steps, equality constraints, and inequality constraints, respectively.
$f_0(\cdot) $ is assumed to be strictly convex and continuously differentiable in $x_s, u_s\ \forall s$, without loss of generality.\footnote{If the original objective does not satisfy the assumption, one can introduce a new variable (say $\tau$) and turning the non-convex objective into a non-convex inequality constraint bounded above by $\tau$.}

\begin{problem}[Non-convex Optimal Control Problem]
\label{prob:OCP}
\begin{align}
\begin{aligned}
& \underset{x,u}{\min}
& & f_0(x,u) \\
&\ \mathrm{s.t.}
& &     x_{s+1} = f_s(x_s, u_s),    \quad s=1,2,...,N-1,\\
& & &   g_i(x,u) = 0 ,    \quad i=1,2,...,p,\\
& & &   h_j(x,u) \leq 0 , \quad j=1,2,...,q,\\
& & &   g_{\mathrm{affine}}(x,u) = 0,\quad h_{\mathrm{cvx}}(x,u) \leq 0
\end{aligned}\nonumber
\end{align}
where
$x = [x_0^\top , x_1^\top, ..., x_N^\top]^\top $,
$u = [u_0^\top , u_1^\top, ..., u_{N-1}^\top]^\top $.
\end{problem}

Defining $z = [x^\top, u^\top]^\top \in\R^{n_z} $, where
$n_z=n_x(N+1) + n_uN$,
\cref{prob:OCP} can be cast as a non-convex optimization problem given in \cref{prob:NLP}, where
$g_i(\cdot) $ incorporates the dynamical constraints $x_{s+1} = f_s(x_s, u_s)$.

\begin{problem}[Non-convex Optimization Problem]
\label{prob:NLP}
\begin{align}
\begin{aligned}
& \underset{z}{\min}
& & f_0(z) \\
&\ \mathrm{s.t.}
& &     g_i(z) = 0 ,    \quad i=1,2,...,p+Nn_x,\\
& & &   h_j(z) \leq 0 , \quad j=1,2,...,q, \\
& & &   g_{\mathrm{affine}}(z) = 0,\quad h_{\mathrm{cvx}}(z) \leq 0
\end{aligned}\nonumber
\end{align}
\end{problem}

Define $g_{\mathrm{all}}(\cdot)$ and $h_{\mathrm{all}}(\cdot)$ to represent all the equality and inequality constraints in \cref{prob:NLP} as:
$g_{\mathrm{all}}(z) = [g(z)^\top, g_{\mathrm{affine}}(z)^\top]^\top $ and
$h_{\mathrm{all}}(z) = [h(z)^\top, h_{\mathrm{cvx}}(z)^\top]^\top $,
where $g = [g_1, ..., g_{p+N n_x}]^\top $ and $h = [h_1, ..., h_q]^\top $ are vectorized non-affine and non-convex constraint functions.

Relevant conditions for a local minimum of \cref{prob:NLP} are provided in \cref{theorem:FONC} 
and \cref{theorem:SOSC}.
These Theorems are obtained by introducing the Lagrangian function $\mathcal{L}(\cdot) $ with Lagrange multiplier vectors $\lambda$ and $\mu $ of appropriate size as:
\begin{align}
\mathcal{L}(z, \lambda, \mu) = 
f_0(z) + \lambda \cdot g_{\mathrm{all}}(z) + \mu \cdot h_{\mathrm{all}}(z)
\label{eq:LagrangianFunction}
\end{align}
and then applying Theorem 12.1 of \cite{Nocedal2006} and Theorem 12.6 of \cite{Nocedal2006} to \cref{prob:NLP}, respectively.
$(\cdot)$ is the dot product operator.

\begin{theorem}[First-order Necessary Conditions]
\label{theorem:FONC}
Suppose that $z_* $ solves \cref{prob:NLP} and that the linear independent constraint qualification (LICQ) holds at $z_* $.
Then there exist Lagrange multiplier vectors $\lambda_*$ and $\mu_* $ such that the KKT conditions are satisfied, i.e.,
$\nabla_z \mathcal{L}(z_*, \lambda_*, \mu_*) = 0 $,
$[g_\mathrm{all}(z_*)]_i = 0 $,
$[h_\mathrm{all}(z_*)]_j \leq 0 $,
$[\mu_*]_j \geq 0 $, and
$[\mu_*]_j [h_\mathrm{all}(z_*)]_j = 0,\ \forall i,j $.
\end{theorem}

\begin{theorem}[Second-order Sufficient Conditions]
\label{theorem:SOSC}
Suppose that for some feasible point $z_* $ of \cref{prob:NLP}, there exist $\lambda_*$ and $\mu_* $ that satisfy the KKT conditions given in \cref{theorem:FONC} and that $\nabla_{zz}^2 \mathcal{L}(z_*, \lambda_*, \mu_*)$ is positive definite on the plane tangent to the constraints, i.e., $v^\top \nabla_{zz}^2 \mathcal{L}(z_*, \lambda_*, \mu_*) v > 0,\ \forall v \in \{v\neq 0 \mid v^\top \nabla_z g_{\mathrm{all}}(z_*) =0 , v^\top \nabla_z h_{\mathrm{active}}(z_*) =0\} $, where $h_{\mathrm{active}} $ denotes the active inequality constraint vector.
Then $z_* $ is a strict local solution for \cref{prob:NLP}.
\end{theorem}

\noindent
\textbf{Assumption 1.}
A local solution $z_*$ of \cref{prob:NLP} together with a unique set of multiplier vectors $\lambda_* $ and $\mu_*$ satisfies the standard second-order sufficient condition for constrained optimization given in \cref{theorem:SOSC}.
Around $z_*$, $g_{\mathrm{all}}(\cdot)$ and $h_{\mathrm{all}}(\cdot) $ are continuously differentiable and satisfy LICQ.

\subsection{Augmented Lagrangian Method}
\label{sec:augLag}

The augmented Lagrangian method \cite{Bertsekas1982} augments the Lagrangian function \cref{eq:LagrangianFunction} as:
\begin{align}
\begin{aligned}
\mathcal{L}_w(z, \lambda, \mu) = 
\mathcal{L}(\cdot) + 
\frac{w}{2} g_{\mathrm{all}} \cdot g_{\mathrm{all}}
+ 
\frac{w}{2} [h_{\mathrm{all}}]_+ \cdot [h_{\mathrm{all}}]_+
\end{aligned}
\label{eq:augLagrangian}
\end{align}
where
$w\in\R$, and
$[x]_+ = \max\{0, x\} $, working element-wise if $x$ is a vector.
This work takes advantage of the property of the augmented Lagrangian method that guarantees the convergence of the variable $z$ and multipliers $\lambda,\mu $ to the optimum, $z_*,\lambda_*, \mu_* $, even if the minimization of $\mathcal{L}_w(\cdot)$ is inexact at each iteration, provided that a few assumptions are met.
\cref{lemma:ALconvergence} gives a summary of this favorable property and the required assumptions in a form tailored to \cref{prob:NLP}.
The superscript $(k)$ denotes the quantities at $k$-th iteration.

\begin{lemma}[Augmented Lagrangian Convergence with Inexact Minimization]
\label{lemma:ALconvergence}
Suppose that for \cref{prob:NLP}, a sequence of $\{{z}_*^{(k)} \} $ satisfies $\norm{\nabla_z \mathcal{L}_{w^{(k)}}({z}_*^{(k)}, \lambda^{(k)}, \mu^{(k)})}_2 \leq \delta^{(k)} $ where ${\delta^{(k)}} \to 0 $ and $\{\lambda^{(k)}, \mu^{(k)}, w^{(k)} \}$ are updated as:
\begin{subequations}
\begin{align}
\lambda^{(k+1)} 	&= \lambda^{(k)} + w^{(k)}g_{\mathrm{all}}({z}_*^{(k)}),\\
\mu^{(k+1)} 		&= [\mu^{(k)} + w^{(k)}h_{\mathrm{all}}({z}_*^{(k)})]_{+}, \\
w^{(k+1)} 			&= \beta w^{(k)} \quad (\beta > 1)
\label{eq:penaltyUpdate}
\end{align}
\label{eq:ALupdate}%
\end{subequations}
where $\{\lambda^{(k)}, \mu^{(k)} \} $ are bounded.
Then, $w^{(k)}$ eventually exceeds a threshold $w_*$ that gives $\nabla_{zz}^2 \mathcal{L}_{w_*} ({z}_*^{(k)}, \lambda^{(k)}, \mu^{(k)}) \succ 0 $, and any sequence $\{{z}_*^{(k)}, \lambda^{(k)}, \mu^{(k)} \}$ globally converges to a local optimum of \cref{prob:NLP}, $\{z_*, \lambda_*, \mu_* \} $.
\end{lemma}

\begin{proof}
The proof is by applying Propositions 2.14, 3.1, and 3.2 of \cite{Bertsekas1982} to \cref{prob:NLP} (see \cite{Bertsekas1976} for an explicit discussion about the global convergence) under Assumption 1.
\end{proof}

Noting that ${z}_*^{(k)}$ represents an approximate minimizer of $\mathcal{L}_{w^{(k)}}(\, \cdot\, , \lambda^{(k)}, \mu^{(k)})$, \cref{lemma:ALconvergence} clarifies that inexact minimization at each augmented Lagrangian iteration must be asymptotically exact, i.e., $\norm{\nabla_z \mathcal{L}_w(\cdot)}_2 \to 0 $ as $k\to \infty$.


\section{The Proposed Algorithm: \texttt{SCvx*}}
\label{sec:algorithm}
This section presents the proposed algorithm, \texttt{SCvx*}.
The convergence analysis of \texttt{SCvx*} is given in \cref{sec:convergenceAnalysis}.

\subsection{Non-convex Penalty Problem with Augmented Lagrangian}

While the augmented Lagrangian function \cref{eq:augLagrangian} is introduced from the viewpoint of primal-dual formalism, it can be also viewed from a penalty method standpoint.
Adopting this viewpoint, \cref{eq:augLagrangian} can be equivalently expressed as 
$\mathcal{L}_w(z, \lambda, \mu) = f_0(z) + P(g_{\mathrm{all}}(z), h_{\mathrm{all}}(z), w, \lambda , \mu )$, where $P(g,h,w,\lambda,\mu)$ denotes the penalty function defined as:
\begin{align}
\begin{aligned}
&P(\cdot)
=
\lambda \cdot g +\frac{w}{2} g \cdot g
+
\mu \cdot [h]_+ + \frac{w}{2} [h]_+\cdot[h]_+
\end{aligned}
\label{eq:penaltyAL}
\end{align}

With the penalty function of the form given by \cref{eq:penaltyAL}, we now formulate our non-convex penalty problem based on \cref{prob:NLP}.
As our algorithm is based on SCP, our penalty problem penalizes the violations of non-convex constraints only ($\because$ convex constraints are imposed in each convex programming);
hence, redefine the Lagrange multipliers as 
\begin{align}
\lambda = [\lambda_1, \lambda_2, ..., \lambda_{p+Nn_x}]^\top,\ 
\mu = [\mu_1, \mu_2, ..., \mu_q]^\top \geq 0 .
\end{align}
This leads to our non-convex penalty problem, \cref{prob:NLP_p}.

\begin{problem}[Non-convex penalty problem with AL]
\label{prob:NLP_p}
\begin{align}
\begin{aligned}
&\underset{z}{\min}
&& 
J(z)  =
f_0(z) +
P(g(z), h(z), w,\lambda,\mu)
\\ 
&\quad \mathrm{s.t.}
&& 	g_{\mathrm{affine}}(z) = 0,\ 
	h_{\mathrm{cvx}}(z) \leq 0
\end{aligned}\nonumber
\end{align}
where $g = [g_1, g_2, ..., g_{p+N n_x}]^\top $ and $h = [h_1, h_2, ..., h_q]^\top $.
\end{problem}

\subsection{Convex Penalty Problem with Augmented Lagrangian}
\label{sec:ALconvexPenaltyProblem}

\cref{prob:NLP_p} is clearly non-convex due to the nonlinearity and non-convexity of $g(\cdot) $ and $h(\cdot) $.
To solve the problem via SCP, we linearize them about a reference variable $\bar{z} $ at each iteration, which yields
$\widetilde{g} = 0 $ and
$\widetilde{h} \leq 0 $, where
\begin{align}
\begin{aligned}
\widetilde{g}(z) &= g(\bar{z}) + \nabla_z g({\bar{z}}) \cdot (z - \bar{z}) ,
\\
\widetilde{h}(z) &= h(\bar{z}) + \nabla_z h({\bar{z}}) \cdot (z - \bar{z}) 
\end{aligned}
\label{eq:linearization}
\end{align}
However, imposing $\widetilde{g} = 0 $ and $\widetilde{h} \leq 0 $ in the convex subproblem can lead to the issue of \textit{artificial infeasibility} \cite{Mao2016}, and hence these linearized constraints are relaxed as follows:
\begin{align}
\begin{aligned}
\widetilde{g}(z) &= \xi
,\quad
\widetilde{h}(z) \leq \zeta
\end{aligned}
\label{eq:linearConstraints}
\end{align}
where $\xi \in\R^{p+Nn_x} $, $\zeta \in\R^q$, and $\zeta \geq 0 $.
Although the original \texttt{SCvx} literature \cite{Mao2016,Mao2019a} introduces \textit{virtual control} and \textit{virtual buffer} terms separately, the former is naturally incorporated in $\xi$.
It is easy to verify this; 
noting that linearized dynamical constraints are given by
\begin{align}
\begin{aligned}
x_{s+1} = A_s x_s + B_s u_s + c_s + E_s \xi,
\quad s=1,2,...,N
\end{aligned}
\label{eq:linearDynamics}
\end{align}
where $E_s\in\R^{n_x\times p + N n_x} $ extracts the virtual control term at $s$-th time instance from $\xi$, and
\begin{align}
\begin{aligned}
A_s =& \nabla_x f_s({\bar{x}_s, \bar{u}_s})
,\ 
B_s = \nabla_u f_s({\bar{x}_s, \bar{u}_s})
,\\ 
c_s =& f_s(\bar{x}_s, \bar{u}_s) - A_s \bar{x}_s - B_s \bar{u}_s
\end{aligned}
\end{align}
it is clear that \cref{eq:linearDynamics} can be incorporated into $\widetilde{g}(z) = \xi$.

On the other hand, the linearization and constraint relaxation lead to another issue called \textit{artificial unboundedness} \cite{Mao2019a}.
To avoid this, we impose a constraint on the variable update magnitude with a trust region bound $r>0$, given by
\begin{align}
\norm{\bar{z} - z}_{\infty} \leq r
\label{eq:trustRegion}
\end{align}
The trust region method is common in many algorithms for nonlinear programming \cite{Nocedal2006}.
This prevents the optimizer from exploring the solution space ``too far'' from $\bar{z}$.

\cref{prob:CVX_p} gives the convex subproblem at each iteration.
\cref{prob:CVX_p} is convex in $z,\xi,\zeta$ since $[x]_+^2 = (\max\{0, x\})^2 $, which appears in \cref{eq:penaltyAL}, is convex in $x\in\R$.

\begin{problem}[Convex penalty subproblem with AL]
\label{prob:CVX_p}
\begin{align}
\begin{aligned}
&\underset{z,\xi,\zeta}{\min}
&& 
L (z,\xi,\zeta) = 
f_0(z)
+
P (\xi,\zeta, w, \lambda, \mu)
\\ 
&\mathrm{s.t.}
\ 
&& 	\widetilde{g}(z) = \xi,\ 
	\widetilde{h}(z) \leq \zeta,\
	\zeta \geq 0,\\ 
&&& \norm{\bar{z} - z}_{\infty} \leq r,\ 
	g_{\mathrm{affine}}(z) = 0,\ 
	h_{\mathrm{cvx}}(z) \leq 0
\end{aligned}\nonumber
\end{align}
\end{problem}

\subsection{\texttt{SCvx*} Algorithm}

We are now ready to present the proposed \texttt{SCvx*} algorithm.
\cref{alg:SCvx*} summarizes \texttt{SCvx*}.
The key steps of \cref{alg:SCvx*} are discussed in the rest of this section.

\begin{algorithm}[tb]
\caption{\texttt{SCvx*}}
\label{alg:SCvx*}
\textbf{Input}: $\bar{z}^{(1)}, r^{(1)}, w^{(1)}, \epsilon_{\mathrm{opt}}, \epsilon_{\mathrm{feas}},
\rho_0, \rho_1, \rho_2, \alpha_1,\alpha_2,\beta,\gamma$
\begin{algorithmic}[1]
\State 
$k= 1$,
$\Delta J^{(0)} = \chi^{(0)} = \delta^{(1)} = \infty $,
$\lambda^{(1)} = \mu^{(1)} = 0 $

\While {$\abs{\Delta J^{(k-1)}} > \epsilon_{\mathrm{opt}} $ or $\chi^{(k-1)} > \epsilon_{\mathrm{feas}} $}
\label{line:convergenceCriterion}
\State $\{\widetilde{g}^{(k)}, \widetilde{h}^{(k)}\} \gets $ derived via \cref{eq:linearization} at $\bar{z}^{(k)} $
\State 
	$\{z^{(k)}_*, \xi^{(k)}_*, \zeta^{(k)}_*\} \gets  $ solve \cref{prob:CVX_p}
\State
	$\{\Delta J^{(k)}, \Delta L^{(k)}, \chi^{(k)}\} \gets $ \cref{eq:optimalityFeasibility}

\If{$\Delta L^{(k)}=0$}
\State $\rho^{(k)} \gets 1$
\Else
\State $\rho^{(k)} \gets \Delta J^{(k)} / \Delta L^{(k)} $
\EndIf

\State
$\{\bar{z}, w, \lambda, \mu, \delta \}^{(k+1)} \gets  \{\bar{z}, w, \lambda, \mu , \delta\}^{(k)} $
\Comment{default}

\If{$\rho^{(k)} \geq \rho_0 $} \Comment{accept the step}
\label{line:iterationAcceptance}
\State 
$\bar{z}^{(k+1)}\gets z^{(k)}_* $
\Comment{solution update}

\If {\cref{eq:LagUpdateCond} is satisfied}
\label{line:ALupdateCriterion}
\State
$\{\lambda, \mu, w \}^{(k+1)} \gets $ \cref{eq:ALupdate}
\Comment{\small{multiplier update}}
\State
$\delta^{(k+1)} \gets $  \cref{eq:stationaryConditionUpdate}
\label{line:KKTupdate}
\Comment{\small{stationarity tol. update}}
\EndIf

\EndIf

\State
$r^{(k+1)}\gets $ \cref{eq:TRupdate} \Comment{trust region update}

\State $k\gets k+1$
\EndWhile
\State \Return $(z^{(k)}_*, \lambda^{(k)}, \mu^{(k)})$
\end{algorithmic}
\end{algorithm}

\subsubsection{Successive linearization}

Let us compactly express the penalty function \cref{eq:penaltyAL} at $k$-th iteration as:
\begin{align}
\begin{aligned}
P^{(k)}(g, h) &\triangleq
P(g, h, w^{(k)},\lambda^{(k)},\mu^{(k)})
\end{aligned}
\end{align}
Likewise, the penalized objectives of \cref{prob:NLP_p,prob:CVX_p} at $k$-th iteration are expressed as:
\begin{align}
\begin{aligned}
J^{(k)} (z)  &\triangleq
f_0(z) +
P^{(k)}(g(z), h(z))
\\
L^{(k)} (z, \xi, \zeta)  &\triangleq
f_0(z) +
P^{(k)}(\xi, \zeta)
\end{aligned}
\end{align}

Given a user-provided initial reference point $\bar{z}^{(1)} $, the linearization process follows \cref{sec:ALconvexPenaltyProblem}, which instantiates \cref{prob:CVX_p} at each iteration.
\cref{prob:CVX_p} is solved to convergence, yielding the solution at $k$-th iteration, $z^{(k)}_*,\xi^{(k)}_*,\zeta^{(k)}_*$.

Every time after \cref{prob:CVX_p} is solved, \texttt{SCvx*} calculates:
\begin{subequations}
\begin{align}
\Delta J^{(k)} &= J^{(k)}(\bar{z}^{(k)}) - J^{(k)}(z^{(k)}_*) 
\label{eq:actualReduction}
\\
\Delta L^{(k)} &= J^{(k)} (\bar{z}^{(k)}) - L^{(k)} (z^{(k)}_*,\xi^{(k)}_*,\zeta^{(k)}_*)
\label{eq:predictedReduction}
\\
\chi^{(k)} &= \norm{g(z^{(k)}_*), [h(z^{(k)}_*)]_{+}}_{2}
\label{eq:infeasibility}
\end{align}
\label{eq:optimalityFeasibility}%
\end{subequations}
where $\Delta J^{(k)},\Delta L^{(k)}$, and $\chi^{(k)}$ represent the actual cost reduction, predicted cost reduction, and the infeasibility.

\subsubsection{Step acceptance}

After solving \cref{prob:CVX_p}, \texttt{SCvx*} accepts the solution and updates $\bar{z}^{(k)} $ if a certain criterion is met.
With $\rho_0\in(0,1) $, the acceptance criterion is given by
\begin{align}
\begin{aligned}
\rho_0 \leq \rho^{(k)},
\quad
\rho^{(k)} = 
{\Delta J^{(k)}}/{\Delta L^{(k)}} 
\label{eq:acceptanceCriterion}
\end{aligned}
\end{align}
where
$\rho^{(k)}$ measures the relative decrease of the objective;
an iteration is accepted only if $\rho^{(k)}$ is greater than $\rho_0$, which helps avoid accepting bad steps (e.g., those which do not improve the non-convex objective).
This criterion is based on the original \texttt{SCvx} \cite{Mao2016}, but not exactly the same;
this point is made precise in the following remark.

\begin{remark}
The definition of $\Delta L^{(k)} $ in \cref{eq:predictedReduction} is different from \texttt{SCvx} \cite{Mao2019a,Mao2016}.
As \texttt{SCvx} considers a fixed penalty weight, their definition of $\Delta L$ with our notation corresponds to $ J^{(k-1)}(\bar{z}^{(k)}) - L^{(k)} (z^{(k)}_*,\xi^{(k)}_*,\zeta^{(k)}_*)$, which is not always non-negative because $J^{(k)}(\bar{z}^{(k)}) \neq J^{(k-1)}(\bar{z}^{(k)}) $.
With the careful definition of $\Delta L^{(k)} $ as in \cref{eq:predictedReduction}, a key result $\Delta L^{(k)} \geq 0 $ is guaranteed in \texttt{SCvx*}, as proved in \cref{lemma:positivePredictedCost}.
\end{remark}

\subsubsection{Lagrange multiplier update}

Although the multipliers $\lambda$ and $\mu$ are fixed in each convex subproblem, they must be updated to march toward the convergence of \cref{prob:NLP}.
\texttt{SCvx*} updates $\lambda $ and $\mu$ when the current iteration is accepted \textit{and} the following condition is met:
\begin{align}
\begin{aligned}
\abs{\Delta J^{(k)}} < \delta^{(k)},
\label{eq:LagUpdateCond}
\end{aligned}
\end{align}
where $\delta^{(k)} \in\R $ is updated such that $\delta^{(k)} \to 0 $ as $k\to\infty$.
The motivation behind this criterion is to satisfy the asymptotically exact minimization requirement clarified in \cref{lemma:ALconvergence}.
A simple design for updating $\delta^{(k)}$ is:
\begin{align}
\delta^{(k+1)} = 
\begin{cases}
\abs{\Delta J^{(k)}} & \mathrm{if}\ \delta^{(k)} = \infty \\
\gamma \delta^{(k)}  & \mathrm{otherwise} \quad \quad (\gamma \in(0,1) )
\end{cases}
\label{eq:stationaryConditionUpdate}
\end{align}
when \cref{eq:LagUpdateCond} is met.
Any other scheme than \cref{eq:stationaryConditionUpdate} may be used as long as it satisfies $\delta^{(k)} \to 0 $ as $k\to\infty$.

Every time when \cref{eq:LagUpdateCond} is met, \texttt{SCvx*} updates $w^{(k)}$, $\lambda^{(k)}$, $\mu^{(k)} $ using \cref{eq:ALupdate}, where $g_{\mathrm{all}}(\cdot)$ and $h_{\mathrm{all}}(\cdot) $ must be replaced by $g(\cdot) $ and $h(\cdot)$.
\cref{sec:convergenceProof} shows that this scheme ensures satisfying the convergence conditions in \cref{lemma:ALconvergence}.
A stricter condition than \cref{eq:LagUpdateCond} is also possible to guarantee the convergence rate, as shown in \cref{sec:convergenceRate}.

\subsubsection{Trust region update}

The trust region radius $r$ plays an important role in preventing artificial unboundedness.
$\rho^{(k)}$ in \cref{eq:acceptanceCriterion} is used to quantify the quality of the current radius $r^{(k)} $.
Like original \texttt{SCvx} \cite{Mao2016}, given the user-defined initial radius $r^{(1)}>0 $ and thresholds $\rho_1$, $\rho_2 \in\R $ that satisfy $\rho_0 < \rho_1 < \rho_2 $, \texttt{SCvx*} updates $r^{(k)} $ as follows:
\begin{align}
r^{(k+1)} =
\begin{cases}
\max\{r^{(k)}/\alpha_1, r_{\min}\} & \mathrm{if}\ \rho^{(k)} < \rho_1 \\
r^{(k)} 			& \mathrm{elseif}\ \rho^{(k)} < \rho_2 
\\
\min\{\alpha_2 r^{(k)}, r_{\max}\} 	& \mathrm{else}
\end{cases}
\label{eq:TRupdate}
\end{align}
where $\alpha_1>1$ and $\alpha_2>1$ determine the contracting and enlarging ratios of $r^{(k)} $, respectively, and $0 < r_{\min} < r_{\max} $.
Although the convergence proof in \cref{sec:convergenceProof} does not require $r^{(k)} $ be bounded from above in theory, \texttt{SCvx*} implements the upper bound $r_{\max}$ for numerical stability.

\subsubsection{Convergence check}

\texttt{SCvx*} detects the convergence to \cref{prob:NLP} and terminates the iteration if:
\begin{align}
\abs{\Delta J^{(k)}} \leq \epsilon_{\mathrm{opt}}
\quad \land \quad
\chi^{(k)} \leq \epsilon_{\mathrm{feas}}
\end{align}
where $\epsilon_{\mathrm{opt}},\epsilon_{\mathrm{feas}} \in\R$ are small positive user-defined scalars representing the optimality and feasibility tolerances.


\section{Convergence Analysis}
\label{sec:convergenceAnalysis}

This section presents the convergence analysis of \texttt{SCvx*}.
\cref{sec:convergenceProof} shows the global strong convergence to \cref{prob:NLP} while \cref{sec:convergenceRate} discusses its convergence rate.

\subsection{Convergence}
\label{sec:convergenceProof}

Let us first introduce \cref{lemma:necessaryCondition}.
\begin{lemma}[Local Optimality Necessary Condition]
\label{lemma:necessaryCondition}
If $z_*$ is a local minimizer of $J^{(k)} $ in \cref{prob:NLP_p}, then $z_*$ is a stationary point of $J^{(k)} $ with the current $w^{(k)},\lambda^{(k)},\mu^{(k)} $.
\end{lemma}
\begin{proof}
Apply Theorem 2.2 of \cite{Nocedal2006} to \cref{prob:NLP_p}.
\end{proof}

We then present \cref{lemma:positivePredictedCost}, which states the non-negativity of $\Delta L^{(k)}$ as well as the stationarity of $J^{(k)}$ when $\Delta L^{(k)} = 0 $.
This extends Theorem 3 of \cite{Mao2016} (also Theorem 3.10 of \cite{Mao2019a}) to account for the effect of varying $w^{(k)},\lambda^{(k)},\mu^{(k)} $.

\begin{lemma}
\label{lemma:positivePredictedCost}
The predicted cost reductions $\Delta L^{(k)} $ in \cref{eq:predictedReduction} satisfy $\Delta L^{(k)} \geq 0 $ for all $k $.
Also, $\Delta L^{(k)} = 0 $ implies that the reference point $z = \bar{z}^{(k)}$ is a stationary point of $J^{(k)}$.
\end{lemma}

\begin{proof}
Since $(z^{(k)}_*,\xi^{(k)}_*,\zeta^{(k)}_*) $ solves \cref{prob:CVX_p}, we have
\begin{align}
\begin{aligned}
L^{(k)} (z^{(k)}_*,\xi^{(k)}_*,\zeta^{(k)}_*) \leq 
L^{(k)} (\bar{z}^{(k)},g(\bar{z}^{(k)}),h(\bar{z}^{(k)}))
\\
=
f_0(\bar{z}^{(k)}) + P^{(k)} (g(\bar{z}^{(k)}),h(\bar{z}^{(k)}))
=
J^{(k)}(\bar{z}^{(k)})
\end{aligned}
\end{align}
Thus, it implies that $\Delta L^{(k)} \geq 0 $ for all $k $ and that $\Delta L^{(k)} = 0 $ holds
if and only if $z^{(k)}_* = \bar{z}^{(k)}$.
From this, $\Delta L^{(k)} = 0 $ implies that $z = \bar{z}^{(k)}$ is a local minimizer of $J^{(k)}$, and hence, from \cref{lemma:necessaryCondition}, a stationary point of $J^{(k)}$.
\end{proof}

\cref{lemma:positivePredictedCost} is a key for \texttt{SCvx*} to inherit two favorable aspects of the original \texttt{SCvx} algorithm, namely, (1) the assured acceptance of iteration and (2) the assured stationarity of limit points.
\cref{lemma:iterationAcceptance,lemma:limitPoint} clarify these two aspects in the context of \texttt{SCvx*} by extending those of \texttt{SCvx}.

\begin{lemma}
\label{lemma:iterationAcceptance}
The \texttt{SCvx*} iterations are guaranteed to be accepted (i.e., Line \ref{line:iterationAcceptance} is satisfied) within a finite number of iterations after an iteration is rejected.
\end{lemma}
\begin{proof}
The proof is straightforward by combining \cref{lemma:positivePredictedCost,lemma:necessaryCondition} and the proof for Lemma 3 of \cite{Mao2016} (or Lemma 3.11 of \cite{Mao2019a}), where the generalized differential and the generalized directional derivative can be replaced with the gradient and directional derivative ($\because$ unlike \texttt{SCvx}, the penalty function of \texttt{SCvx*} is differentiable due to the formulation based on the augmented Lagrangian method).
\end{proof}

\begin{lemma}
\label{lemma:limitPoint}
A sequence $\{z^{(k)}_* \}$ generated by \texttt{SCvx*} when the Lagrange multipliers and penalty weight are fixed is guaranteed to have limit points, and any limit point $\hat{z} $ is a stationary point of \cref{prob:NLP_p}.
\end{lemma}
\begin{proof}
The proof is straightforward by combining \cref{lemma:positivePredictedCost,lemma:iterationAcceptance} and the proof for Theorem 4 of \cite{Mao2016} (or Theorem 3.13 of \cite{Mao2019a}), where note from \cref{eq:TRupdate} that $r^{(k)} \geq r_{\min} > 0 $.
\end{proof}

Remarkably, \cref{lemma:limitPoint} implies that,
when the values of $\lambda^{(k)}, \mu^{(k)}, w^{(k)}  $ remain fixed, we have
$\Delta J^{(k_i)} \to 0 $
as $i\to \infty$, where $\{z_*^{(k_i)}\} $ is a subsequence of $\{z_*^{(k)}\} $.
This assures the satisfaction of Line \ref{line:ALupdateCriterion}
within a finite (typically a few) number of iterations after $\lambda^{(k)}, \mu^{(k)}, w^{(k)} $ are last updated.
This key property is formally stated in \cref{lemma:guaranteedALupdate}.

\begin{lemma}[]
\label{lemma:guaranteedALupdate}
The \texttt{SCvx*} multipliers and penalty weights are guaranteed to be updated (i.e., Line \ref{line:ALupdateCriterion} is satisfied) within a finite number of iterations after their last update.
\end{lemma}
\begin{proof}
The proof is by contradiction.
Suppose Line \ref{line:ALupdateCriterion} is not satisfied for indefinite number of iterations, i.e., $\abs{\Delta J^{(k)}} \geq \delta^{(k)}$ for $k\to \infty$.
It implies $\lambda^{(k)}, \mu^{(k)}, w^{(k)} $ remain the same for $k\to \infty$.
However, when $\lambda^{(k)}, \mu^{(k)}, w^{(k)} $ remain the same values, there is at least one subsequence with $\Delta J^{(k_i)} \to 0 $ due to \cref{lemma:limitPoint}, which eventually satisfies $\abs{\Delta J^{(k)}} < \delta^{(k)}$ for any $\delta^{(k)}>0$ without requiring infinite $k$.
This contradicts $\abs{\Delta J^{(k)}} \geq \delta^{(k)}$ for $k\to \infty$, and thus implies \cref{lemma:guaranteedALupdate}.
\end{proof}

We are now ready to present the main result of this paper on the convergence property of the \texttt{SCvx*} algorithm.

\begin{theorem}[Global Strong Convergence with Feasibility]
\texttt{SCvx*} achieves global convergence to a feasible local optimum of the original problem, \cref{prob:NLP}.
\end{theorem}
\begin{proof}
Let $\{z^{(k_i)}_* \} $ be a subsequence of $\{z^{(k)}_* \} $ that consists of the iterations where the multipliers are updated;
such subsequences are guaranteed to exist due to \cref{lemma:guaranteedALupdate}.
Then, \cref{eq:stationaryConditionUpdate} ensures $\delta^{(k_i)} > \delta^{(k_{i}+1)} $, and due to \cref{lemma:limitPoint}, $\Delta J^{(k_i)} \to 0 $ and $\delta^{(k_i)} \to 0$ in the limit.
Again due to \cref{lemma:limitPoint}, the limit point is a stationary point of \cref{prob:NLP_p}, satisfying $\nabla_z J^{(k)} = 0 $.
Then, $\nabla_z \mathcal{L}_{w^{(k)}} \to 0 $ also holds in the limit since every $z^{(k)}_*$ satisfies $g_{\mathrm{affine}} = 0 $ and $h_{\mathrm{cvx}} \leq 0 $ within convex programming.
Thus, the \texttt{SCvx*} iteration guarantees $\norm{\nabla_z \mathcal{L}_{w^{(k)}} }_2 \to 0$ in the limit, with the multiplier update \cref{eq:ALupdate}.
Therefore, as $w^{(k)} $ exceeds the threshold $w_*$ given in \cref{lemma:ALconvergence} after finite iterations, \texttt{SCvx*} achieves the global convergence to a feasible optimum of \cref{prob:NLP}.
\end{proof}

Remarks below discuss two key improvements that the \texttt{SCvx*} algorithm provides over the original \texttt{SCvx} algorithm.

\begin{remark}[Feasibility]
\label{remark:feasibility}
The converged solution generated by \texttt{SCvx*} is feasible to \cref{prob:NLP}, while the original \texttt{SCvx} algorithm does not provide such a feasibility guarantee.
\end{remark}

\begin{remark}[Accelerated convergence]
\label{remark:acceleration}
\texttt{SCvx*} iterates not only on the variable $z^{(k)} $ but also on the Lagrange multipliers $\lambda^{(k)} $ and $\mu^{(k)} $, which, besides providing the feasibility guarantee, facilitates the convergence by iteratively improving the multiplier estimate rather than using a fixed value.
\end{remark}

\subsection{Convergence rate}
\label{sec:convergenceRate}

Having the augmented Lagrangian method as the basis of the algorithm facilitates the analysis of the convergence rate of \texttt{SCvx*}.
Based on \cite{Bertsekas1982,Bertsekas1976}, linear or superlinear convergence rate of the Lagrangian multipliers can be achieved when $\delta^{(k)} $ decreases to zero as fast as $\norm{\lambda^{(k)} - \lambda_* }_2/w^{(k)} $.
To achieve this, we may replace \cref{eq:LagUpdateCond} by
\begin{align}
\abs{\Delta J^{(k)}} \leq \min{\{\delta^{(k)}, \eta \chi^{(k)} \}} ,
\label{eq:newALupdateCriterion}
\end{align}
where $\eta$ is a positive scalar.
With this criterion, Proposition 2 of \cite{Bertsekas1976} states that the augmented Lagrange multiplier iteration \cref{eq:ALupdate} converges to $z_*, \lambda_*, \mu_*$ superlinearly if $w^{(k)}\to \infty $, and linearly if $w^{(k)}\to w_{\max} < \infty $, where $w_{\max} \in (0, \infty)$ is the upper bound of the penalty weight.
It must be noted that these convergence rates are about the Lagrange multiplier iteration but not with respect to $k$.

The choice of $\eta$ can be arbitrary to achieve the above convergence rate in theory.
A simple yet effective approach is to initialize $\eta$ by $\infty$ at first, and then update it by
$\eta \gets \abs{\Delta J^{(k)}} / \chi^{(k)}$
when \cref{eq:newALupdateCriterion} is met for the first time.

Here, we must ensure that \texttt{SCvx*} retains the favorable property of guaranteed multiplier update (\cref{lemma:guaranteedALupdate}) under the stricter condition \cref{eq:newALupdateCriterion}.
\cref{lemma:guaranteedALupdateNew} addresses this.
Once the convergence in $\lambda,\mu$ is achieved, then $\lambda,\mu$ will not be updated anymore while $z^{(k)} $ converges to a feasible local minimum of \cref{prob:NLP} due to \cref{lemma:ALconvergence,lemma:limitPoint}.

\begin{lemma}[]
\label{lemma:guaranteedALupdateNew}
Suppose that \cref{eq:newALupdateCriterion} instead of \cref{eq:LagUpdateCond} is used for the multiplier update criterion.
Then, until the convergence in $\lambda$ and $\mu$ is achieved, the \texttt{SCvx*} iteration guarantees that the values of $\lambda$ and $\mu$ are updated within a finite number of iterations after their last update.
\end{lemma}

\if\shortOrFull1 
\begin{proof}
See \cite{Oguri2023a} for the proof.
\end{proof}
\fi

\if\shortOrFull2 
\begin{proof}
For conciseness, the proof is focused on problems with equality constraints only, as any inequality constraints can be converted to equality constraints by introducing dummy variables without changing the results in the augmented Lagrangian framework (see Section 3.1 of \cite{Bertsekas1982}).
Thus, $\mu$ and $h$ are not explicitly considered in this proof.

It is clear from \cref{lemma:limitPoint} that the claim is true if $\chi^{(k)} > 0 $ holds until the convergence in $\lambda$ is achieved.
It is also clear that $\lambda^{(k)} \neq \lambda_* $ until the convergence in $\lambda$ is achieved.
Thus, let us show $\chi^{(k)} > 0 $ when $\lambda^{(k)} \neq \lambda_* $ by contradiction.

Suppose that there exists certain $\lambda^{(k)} (\neq \lambda_*) $ such that lead to $\chi^{(k)} = 0 $.
For $z^{(k)}_* $ that solves \cref{prob:CVX_p}, it is clear from \cref{eq:infeasibility} that $\chi^{(k)} = 0 $ if and only if $g(z^{(k)}_*) = 0 $.
Since $\chi^{(k)} = 0 $, \cref{eq:newALupdateCriterion} is not satisfied, and hence the values of $\lambda^{(k)}$ remain fixed until $\Delta J^{(k)} = 0 $ is achieved in the limit.
Due to \cref{lemma:limitPoint}, the limit point is a stationary point of \cref{prob:NLP_p}, satisfying 
$0 = \nabla_z J^{(k)} = \nabla_z f_0 + (\lambda + w g) \cdot \nabla_z g$.
Using $g = 0 $ due to $\chi^{(k)} = 0 $, this leads to 
$0 = \nabla_z f_0 + \lambda \cdot \nabla_z g$,
which implies that the limit point of $\{z^{(k)}_*\} $ is a feasible stationary point of \cref{prob:NLP}, i.e., $\nabla_z \mathcal{L} = 0$, and hence satisfies the KKT conditions of \cref{prob:NLP}, since every $z^{(k)}_*$ also satisfies $g_{\mathrm{affine}} = 0 $ and associated multiplier conditions within convex programming.
This contradicts $\lambda^{(k)} \neq \lambda_* $, and thus $\chi^{(k)} \neq 0 $ by contradiction, implying $\chi^{(k)} > 0 $ because $\chi^{(k)} $ must be non-negative.
\end{proof}
\fi


\section{Numerical Examples}
\label{sec:examples}

This section presents numerical examples to demonstrate \texttt{SCvx*} and compare the performance to \texttt{SCvx}.
Note that \cref{alg:SCvx*} boils down to \texttt{SCvx} by ignoring Lines \ref{line:ALupdateCriterion} to \ref{line:KKTupdate} and replacing \cref{eq:penaltyAL} by an $l_1$ penalty function $P(g,h,w) = w \norm{g}_1 + w \norm{[h]_+}_1 $.
CVX \cite{Grant2014} is used with Mosek \cite{Mosek2017}.

\texttt{SCvx*} parameters commonly used for the two examples are listed in \cref{t:params}.
In each example, $w^{(1)} $ is varied to investigate the performance of \texttt{SCvx*} and \texttt{SCvx} for different penalty weights.
$w_{\max} = 10^{8} $ is set for \texttt{SCvx*} to avoid numerical instability.
If the algorithm does not converge in $100$ iterations, it is terminated and deemed unconverged.

\begin{table}[tb]
\centering
\caption{\texttt{SCvx*} parameters ($\epsilon = \epsilon_{\mathrm{opt}} = \epsilon_{\mathrm{feas}}$)}
\label{t:params}
\begin{tabular}{lcccccccccccc}
\toprule
$ \epsilon $ & $ \{\rho_0, \rho_1, \rho_2\} $ & $\{\alpha_1, \alpha_2, \beta, \gamma\} $ & $\{r^{(1)}, r_{\min}, r_{\max}\} $
\\ \midrule
$10^{-5}$ & $\{0, 0.25, 0.7\} $ & $\{2, 3, 2, 0.9 \}$ & $\{0.1, 10^{-10}, 10\}$
\\ \bottomrule
\end{tabular}
\end{table}

\subsection{Example 1: Simple Problem with Crawling Phenomenon}
The first example is a simple non-convex optimization problem from \cite{Reynolds2020} to demonstrate that \texttt{SCvx*} can also overcome the so-called \textit{crawling phenomenon}, which is known to occur for a class of SCP algorithms.
The non-convex problem from \cite{Reynolds2020} is defined in the form of \cref{prob:NLP} as follows:
\begin{subequations}
\begin{align}
& \underset{-2\leq z \leq 2}{\min}
& & z_1 + z_2 \\
&\ \mathrm{s.t.}
& &     z_2- z_1^4 - 2z_1^3 + 1.2z_1^2 + 2z_1 = 0, \label{eq:ex1-ncvx}
\\
& & &   -z_2 - (4/3)z_1 - 2/3 \leq 0 \label{eq:ex1-cvx}
\end{align}%
\end{subequations}
which is solved by \texttt{SCvx*} and \texttt{SCvx} with various $w^{(1)} $.
The same initial reference point as \cite{Reynolds2020}, $\bar{z}^{(1)} = [1.5, 1.5] $, is used.

\cref{t:results1} summarizes the comparison of \texttt{SCvx*} and \texttt{SCvx} by listing the number of iterations required for convergence with respect to different values of $w^{(1)} $.
``N/A'' indicates non-convergence achieved within the maximum iteration limit ($=100$).
\cref{t:results1} illustrates that \texttt{SCvx*} constantly achieves the convergence regardless of the initial values of $w^{(k)} $;
this is in sharp contrast to the \texttt{SCvx} results, which successfully converge to a feasible local minimum only for the two cases: $w^{(1)} = 10$ and $100 $, emphasizing the sensitivity to the value of $w^{(1)} $ (which is held constant over iterations in \texttt{SCvx}).

\begin{table}[tb]
\centering
\caption{Example 1: Convergence Results (N/A: non-convergence)}
\label{t:results1}
\begin{tabular}{lcccccccccccc}
\toprule
$w^{(1)} $ value & $10^{-1}$ & $10^{0}$ & $10^{1}$ & $10^{2}$ & $10^{3}$ & $10^{4}$ & $10^{5}$ 
\\ \midrule
\texttt{SCvx*} \# ite. & \tb{39} & \tb{33} & \tb{31} & \tr{42} & \tb{40} & \tb{51} & \tb{56}
\\
\texttt{SCvx} \# ite. & \tr{N/A} & \tr{N/A} & \tr{35} & \tb{31} & \tr{N/A} & \tr{N/A} & \tr{N/A}
\\ \bottomrule
\end{tabular}
\end{table}

\begin{figure}[tb]
\centering \subfigure[\label{f:ex1:w1:SCvx*Plot} \texttt{SCvx*} result, 33 iterations]
{\includegraphics[width=0.47\linewidth]{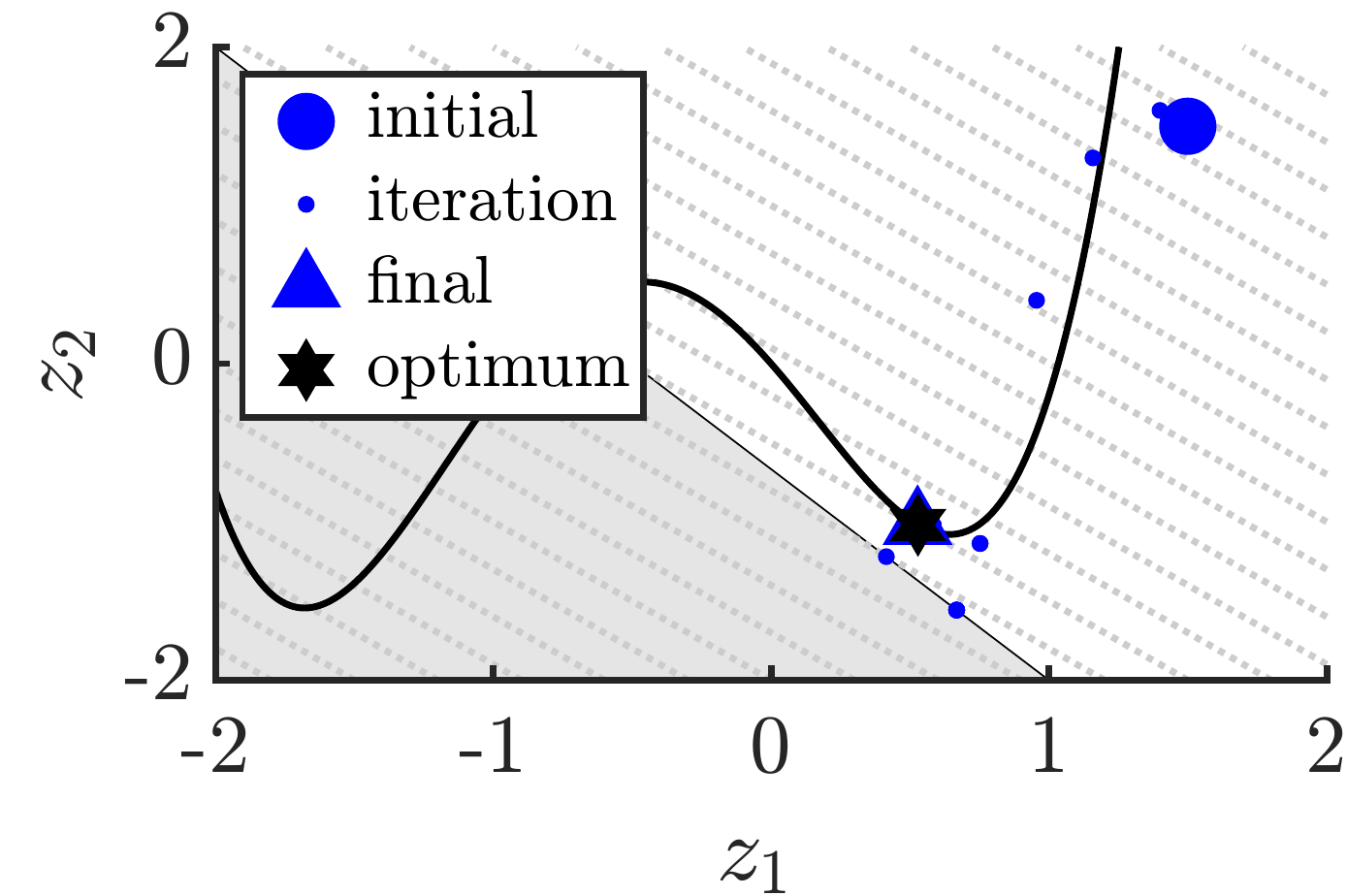}}
\centering \subfigure[\label{f:ex1:w1:SCvxPlot} \texttt{SCvx} result, \textit{not} converged]
{\includegraphics[width=0.47\linewidth]{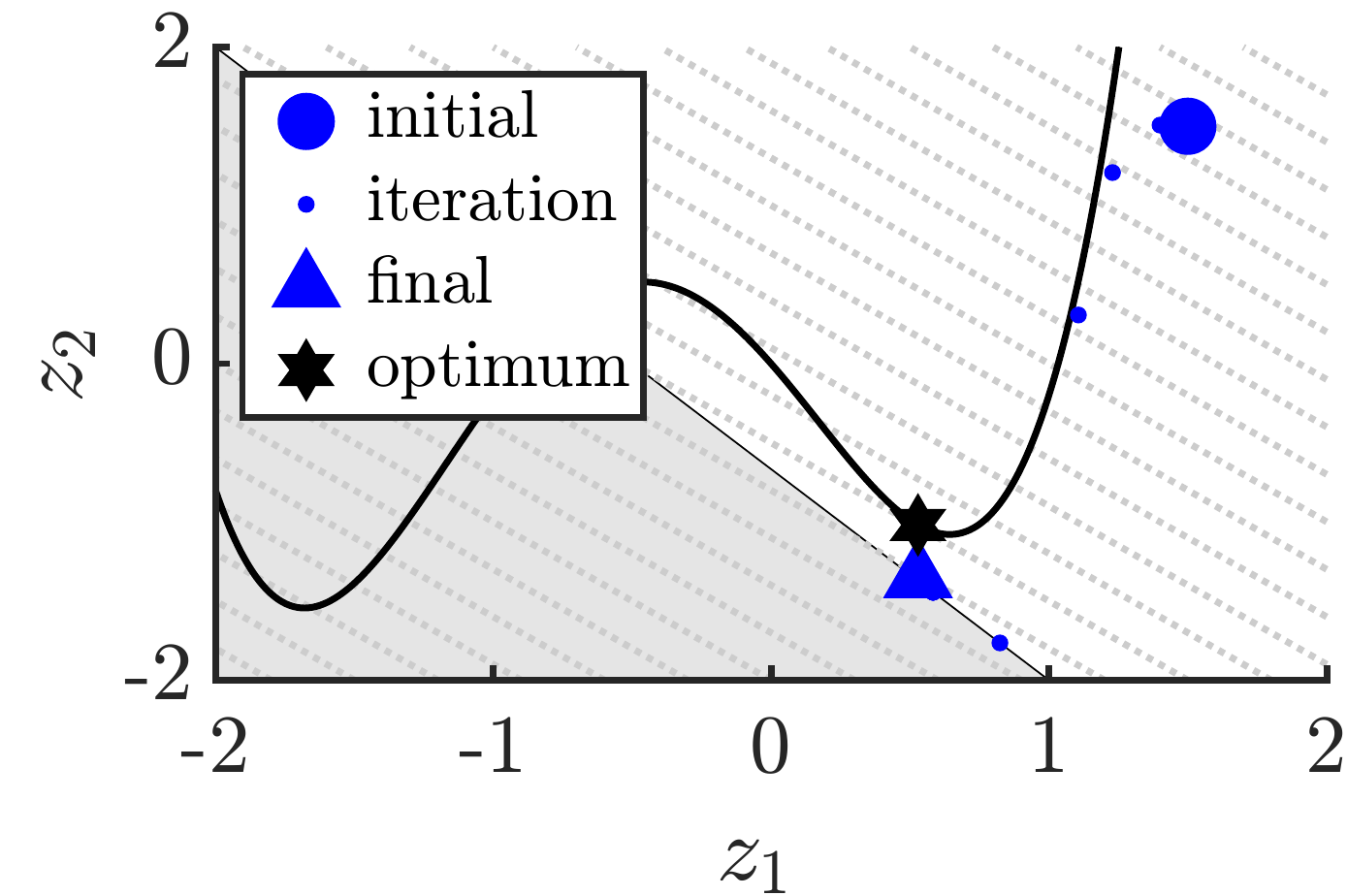}}
	\caption{\label{f:ex1:w1} Example 1 convergence behavior for $w^{(1)}=1 $ in the $z_1$-$z_2$ space. The black curve represents the non-convex equality constraint \cref{eq:ex1-ncvx}, and the gray shaded area is the infeasible area from \cref{eq:ex1-cvx}.}
\end{figure}

\cref{f:ex1:w1} depicts the convergence behavior for $w^{(1)} {=} 1 $.
It clarifies that the \texttt{SCvx*} iteration successfully converges to the optimum 
while \texttt{SCvx} is not able to satisfy the non-convex equality constraint (on the black curve).
This example illustrates two non-converging modes of \texttt{SCvx}:
1) non-improving feasibility for $w^{(1)}{=}10^{-1}, 10^0$; and 2) crawling phenomenon for $w^{(1)}{=}10^{3}, 10^{4}, 10^{5}$.
In contrast, \texttt{SCvx*} effectively addresses these non-converging modes, consistent with the two key improvements stated in \cref{remark:feasibility,remark:acceleration}.

\subsection{Example 2: Quad Rotor Path Planning}
The second example is a quad-rotor non-convex optimal control problem from the original \texttt{SCvx} literature \cite{Mao2019a}.
The problem is defined in the form of \cref{prob:OCP} as follows:
\begin{align}
\begin{aligned}
& \underset{x, u}{\min}
& & \sum_{s=1}^{N} \Gamma_s \Delta t \\
&\ \mathrm{s.t.}
& &     x_{s+1} = x_s + \int_{t_s}^{t_{s+1}}
\begin{bmatrix}
v \\ 
T/m - k_D\norm{v}_2 v + g
\end{bmatrix}
\diff t,
\ \forall s
\\ & & &   
\norm{p_s - p_{\mathrm{obj},j}}_2 \geq R_{\mathrm{obs},j},
\quad j=1,2,
\\ & & &   
x_1 = x_{\mathrm{ini}},\ 
x_N = x_{\mathrm{fin}},\ 
T_1 = T_N = -mg,
\\ & & &   
[1\ 0\ 0] \cdot p_s = 0,\ 
\norm{T_s}_2 \leq \Gamma_s,\ 
T_{\min} \leq \Gamma_s \leq T_{\max},
\\ & & &   
\cos{\theta_{\max}}\Gamma_s \leq [1\ 0\ 0] \cdot  T_s,
\quad \forall s
\end{aligned}
\nonumber
\end{align}
where $p, v\in\R^3$, and $m\in\R$ denote the position, velocity, and mass of the vehicle;
$T\in\R^3 $ is the thrust vector;
$\Gamma\in\R $ represents the thrust magnitude (at convergence);
$g = [-9.81, 0, 0]^\top\, \mathrm{m/s^2} $ is the gravity acceleration;
$k_D = 0.5 $ is the drag coefficient;
$p_{\mathrm{obj},j}$ and $R_{\mathrm{obj},j} $ are the position and radius of $j$-th obstacle (defined the same as \cite{Mao2019a});
$x_{\mathrm{ini}} = [0\, \mathrm{m}, 0\, \mathrm{m}, 0\, \mathrm{m}, 0\, \mathrm{m/s}, 0.5\, \mathrm{m/s}, 0\, \mathrm{m/s}]^\top $ and
$x_{\mathrm{fin}} = [0\, \mathrm{m}, 10\, \mathrm{m}, 0\, \mathrm{m}, 0\, \mathrm{m/s}, 0.5\, \mathrm{m/s}, 0\, \mathrm{m/s}]^\top $ are the initial and final states;
$\{T_{\min}, T_{\max}\} = \{1.0, 4.0\}\, \mathrm{N}$;
$\theta_{\max} = \pi/4 $.
The state and control variables are
$x_s = [r_s^\top, v_s^\top]^\top \in\R^6 $ and 
$u_s = [T_s^\top, \Gamma_s]^\top \in\R^4 $, where the zeroth-order-hold control is used for the discretization, i.e., $u_s = u(t)\ \forall t \in[t_s, t_{s+1}) $.
$t_N = 5.0 $ seconds with $N=31$.
For $\bar{z}^{(1)} $, the straight line that connects $x_{\mathrm{init}}$ and $x_{\mathrm{fin}}$ is used for $x_s$ while $-mg$ and $\norm{mg}_2 $ are used for $T_s $ and $\Gamma_s$, respectively.

\cref{t:results2} summarizes the convergence results for Example 2.
This suggests that the performance of \texttt{SCvx*} and \texttt{SCvx} are similar overall for this example, whereas a key difference is observed that \texttt{SCvx} struggles to converge to a feasible solution when $w^{(1)} = 10^{-1} $.
\texttt{SCvx*} constantly converges to a feasible local minimum irrelevant to the value of $w^{(1)} $.
This property is favorable especially for large-scale problems, where the user may not afford to tune $w^{(1)}$ \cite{Oguri2022f}.
On the other hand, this also provides a reassuring result that, despite the lack of the theoretical feasibility guarantee, \texttt{SCvx} can also perform well and may be good enough for relatively simple, small-scale optimal control problems.

\begin{table}[tb]
\centering
\caption{Example 2: Convergence Results}
\label{t:results2}
\begin{tabular}{lcccccccccccc}
\toprule
$w^{(1)} $ value & $10^{-1}$ & $10^{0}$ & $10^{1}$ & $10^{2}$ & $10^{3}$ & $10^{4}$ & $10^{5}$ 
\\ \midrule
\texttt{SCvx*} \# ite. & \tb{24} & \tr{17} & \tr{14} & \tb{11} & \tb{11} & \tb{11} & \tb{14}
\\
\texttt{SCvx} \# ite. & \tr{N/A} & \tb{9} & \tb{11} & \tr{13} & \tr{14} & \tr{15} & \tr{16}
\\ \bottomrule
\end{tabular}
\end{table}

\begin{figure}[tb]
\centering \subfigure[\label{f:ex2:w1e5:SCvx*Plot} \texttt{SCvx*} result, 11 iterations]
{\includegraphics[width=0.46\linewidth]{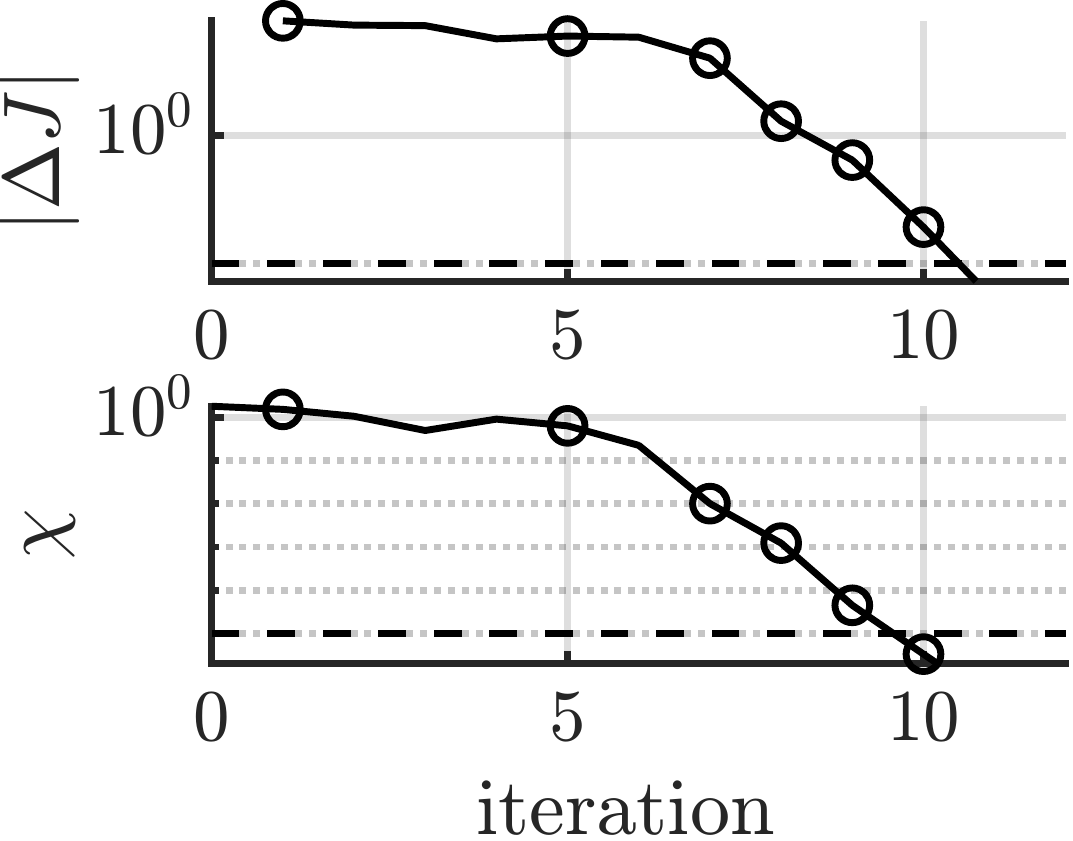}}
\centering \subfigure[\label{f:ex2:w1e5:SCvxPlot} \texttt{SCvx} result, 15 iterations]
{\includegraphics[width=0.46\linewidth]{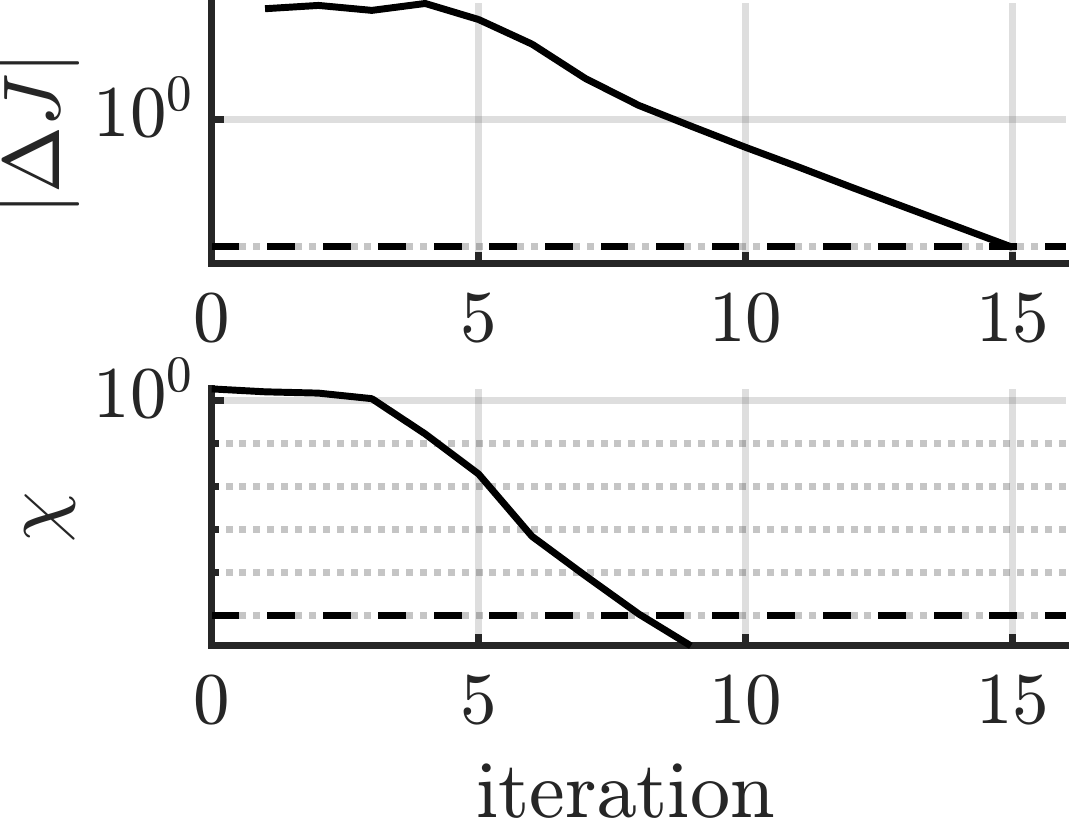}}
	\caption{\label{f:ex2:w1e5} Example 2 convergence behavior for $w^{(1)}=10^{4} $; $y$-axes represent the penalized objective improvement $\abs{\Delta J^{(k)}} $ (top) and the infeasibility $\chi^{(k)} $ (bottom).
	The circles in (a) indicate iterations with multiplier updates.}
\end{figure}

\cref{f:ex2:w1e5} presents the convergence behavior for $w^{(1)} = 10^4 $ in terms of $\abs{\Delta J^{(k)}} $ and $\chi^{(k)} $.
The circles in \cref{f:ex2:w1e5:SCvx*Plot} indicate the iterations when Line \ref{line:ALupdateCriterion} of \cref{alg:SCvx*} is satisfied and the multipliers are updated.
The dashed lines represent the tolerance $\epsilon {=} \epsilon_{\mathrm{opt}} {=} \epsilon_{\mathrm{feas}} (=10^{-5}) $.
While \texttt{SCvx} satisfies the constraints earlier, \texttt{SCvx*} achieves the overall convergence faster, likely due to the iterative estimate of multipliers that balances the progress in optimality and feasibility.

\section{Conclusions}
In this paper, a new SCP algorithm \texttt{SCvx*} is proposed to address the lack of feasibility guarantee in \texttt{SCvx} by leveraging the augmented Lagrangian framework.
Unlike \texttt{SCvx}, which uses a fixed penalty weight over iterations, 
\texttt{SCvx*} iteratively improves both the optimization variables and the Lagrange multipliers, facilitating the convergence.
Inheriting the favorable properties of \texttt{SCvx} and fusing those with the augmented Lagrangian method, \texttt{SCvx*} provides strong global convergence to a \textit{feasible} local optimum of the original non-convex optimal control problems with minimal requirements on the problem form.
The convergence rate of \texttt{SCvx*} is also analyzed, clarifying that linear/superlinear convergence rate with respect to the Lagrange multipliers can be achieved by slightly modifying the algorithm.
These theoretical results are demonstrated via numerical examples.



\begin{thebibliography}{10}

\bibitem{Mao2019a}
Y.~Mao, M.~Szmuk, X.~Xu, and B.~Acikmese, ``Successive {{Convexification}}: {{A
  Superlinearly Convergent Algorithm}} for {{Non-convex Optimal Control
  Problems}},'' {\em arXiv preprint}, Feb. 2019.

\bibitem{Mao2016}
Y.~Mao, M.~Szmuk, and B.~Acikmese, ``Successive convexification of non-convex
  optimal control problems and its convergence properties,'' in {\em {{IEEE}}
  55th {{Conference}} on {{Decision}} and {{Control}} ({{CDC}})},
  pp.~3636--3641, {IEEE}, Dec. 2016.

\bibitem{Acikmese2007}
B.~A{\c c}{\i}kme{\c s}e and S.~R. Ploen, ``Convex {{Programming Approach}} to
  {{Powered Descent Guidance}} for {{Mars Landing}},'' {\em Journal of Guidance
  Control and Dynamics}, vol.~30, no.~5, p.~1353, 2007.

\bibitem{Akmee2011a}
B.~A{\c c}{\i}kme{\c s}e and L.~Blackmore, ``Lossless convexification of a
  class of optimal control problems with non-convex control constraints,'' {\em
  Automatica}, vol.~47, pp.~341--347, Feb. 2011.

\bibitem{Harris2021}
M.~W. Harris, ``Optimal {{Control}} on {{Disconnected Sets Using Extreme Point
  Relaxations}} and {{Normality Approximations}},'' {\em IEEE Transactions on
  Automatic Control}, vol.~66, pp.~6063--6070, Dec. 2021.

\bibitem{Nocedal2006}
J.~Nocedal and S.~J. Wright, {\em Numerical {{Optimization}}}.
\newblock {Springer New York}, 2006.

\bibitem{Simon2013}
D.~Simon, {\em Evolutionary {{Optimization Algorithms}}}.
\newblock {John Wiley \& Sons, Incorporated}, 2013.

\bibitem{Malyuta2022}
D.~Malyuta, T.~P. Reynolds, M.~Szmuk, T.~Lew, R.~Bonalli, M.~Pavone, and
  B.~A{\c c}{\i}kme{\c s}e, ``Convex {{Optimization}} for {{Trajectory
  Generation}}: {{A Tutorial}} on {{Generating Dynamically Feasible
  Trajectories Reliably}} and {{Efficiently}},'' {\em IEEE Control Systems
  Magazine}, vol.~42, pp.~40--113, Oct. 2022.

\bibitem{Bonalli2019b}
R.~Bonalli, A.~Cauligi, A.~Bylard, and M.~Pavone, ``{{GuSTO}}: {{Guaranteed
  Sequential Trajectory}} optimization via {{Sequential Convex Programming}},''
  in {\em International {{Conference}} on {{Robotics}} and {{Automation}}
  ({{ICRA}})}, pp.~6741--6747, {IEEE}, May 2019.

\bibitem{Boyd2004}
S.~Boyd and L.~Vandenberghe, {\em Convex {{Optimization}}}.
\newblock {Cambridge, England}: {Cambridge University Press}, Mar. 2004.

\bibitem{Palacios-Gomez1982}
F.~{Palacios-Gomez}, L.~Lasdon, and M.~Engquist, ``Nonlinear {{Optimization}}
  by {{Successive Linear Programming}},'' {\em Management Science}, Oct. 1982.

\bibitem{Yuille2003}
A.~L. Yuille and A.~Rangarajan, ``The {{Concave-Convex Procedure}}
  ({{CCCP}}),'' in {\em Advances in {{Neural Information Processing Systems}}},
  2003.

\bibitem{Boggs1989}
P.~T. Boggs and J.~W. Tolle, ``A {{Strategy}} for {{Global Convergence}} in a
  {{Sequential Quadratic Programming Algorithm}},'' {\em SIAM J. Numer. Anal.},
  vol.~26, pp.~600--623, June 1989.

\bibitem{Gill2002}
P.~E. Gill, W.~Murray, and M.~A. Saunders, ``{{SNOPT}}: {{An SQP Algorithm}}
  for {{Large-Scale Constrained Optimization}},'' {\em SIAM Journal on
  Optimization}, vol.~12, pp.~979--1006, Jan. 2002.

\bibitem{Mao2021}
Y.~Mao and B.~Acikmese, ``{{SCvx-fast}}: {{A Superlinearly Convergent
  Algorithm}} for {{A Class}} of {{Non-Convex Optimal Control Problems}},''
  {\em arXiv preprint}, pp.~1--22, Nov. 2021.

\bibitem{Bonalli2022a}
R.~Bonalli, T.~Lew, and M.~Pavone, ``Analysis of {{Theoretical}} and
  {{Numerical Properties}} of {{Sequential Convex Programming}} for
  {{Continuous-Time Optimal Control}},'' {\em IEEE Transactions on Automatic
  Control}, pp.~1--16, 2022.

\bibitem{Szmuk2020}
M.~Szmuk, T.~P. Reynolds, and B.~A{\c c}{\i}kme{\c s}e, ``Successive
  {{Convexification}} for {{Real-Time Six-Degree-of-Freedom Powered Descent
  Guidance}} with {{State-Triggered Constraints}},'' {\em Journal of Guidance,
  Control, and Dynamics}, vol.~43, pp.~1399--1413, Aug. 2020.

\bibitem{Bonalli2019a}
R.~Bonalli, A.~Bylard, A.~Cauligi, T.~Lew, and M.~Pavone, ``Trajectory
  {{Optimization}} on {{Manifolds}}: {{A Theoretically-Guaranteed Embedded
  Sequential Convex Programming Approach}},'' in {\em Proceedings of
  {{Robotics}}: {{Science}} and {{Systems}}}, May 2019.

\bibitem{Lew2020}
T.~Lew, R.~Bonalli, and M.~Pavone, ``Chance-{{Constrained Sequential Convex
  Programming}} for {{Robust Trajectory Optimization}},'' in {\em European
  {{Control Conference}} ({{ECC}})}, pp.~1871--1878, {IEEE}, May 2020.

\bibitem{Bertsekas1982}
D.~P. Bertsekas, {\em Constrained {{Optimization}} and {{Lagrange Multiplier
  Methods}}}.
\newblock {Elsevier}, 1982.

\bibitem{Hestenes1969}
M.~R. Hestenes, ``Multiplier and gradient methods,'' {\em Journal of
  Optimization Theory and Applications}, vol.~4, pp.~303--320, Nov. 1969.

\bibitem{Powell1978}
M.~J.~D. Powell, ``Algorithms for nonlinear constraints that use lagrangian
  functions,'' {\em Mathematical Programming}, vol.~14, pp.~224--248, Dec.
  1978.

\bibitem{Oguri2022f}
K.~Oguri and G.~Lantoine, ``Stochastic {{Sequential Convex Programming}} for
  {{Robust Low-thrust Trajectory Design}} under {{Uncertainty}},'' in {\em
  {{AAS}}/{{AIAA Astrodynamics Specialist Conference}}}, 2022.

\bibitem{Oguri2022}
K.~Oguri, G.~Lantoine, and T.~H. Sweetser, ``Robust {{Solar Sail Trajectory
  Design}} under {{Uncertainty}} with {{Application}} to {{NEA Scout
  Mission}},'' in {\em {{AIAA SCITECH Forum}}}, ({Reston, Virginia}), {American
  Institute of Aeronautics and Astronautics}, Jan. 2022.

\bibitem{Bertsekas1976}
D.~P. Bertsekas, ``On {{Penalty}} and {{Multiplier Methods}} for {{Constrained
  Minimization}},'' {\em SIAM J. Control Optim.}, vol.~14, pp.~216--235, Feb.
  1976.

\bibitem{Attouch2013}
H.~Attouch, J.~Bolte, and B.~F. Svaiter, ``Convergence of descent methods for
  semi-algebraic and tame problems: Proximal algorithms, forward\textendash
  backward splitting, and regularized {{Gauss}}\textendash{{Seidel}} methods,''
  {\em Math. Program.}, vol.~137, pp.~91--129, Feb. 2013.

\bibitem{Grant2014}
M.~Grant and S.~Boyd, ``{{CVX}}: {{Matlab Software}} for {{Disciplined Convex
  Programming}}, version 2.1,'' Mar. 2014.

\bibitem{Mosek2017}
{Mosek ApS}, ``The {{MOSEK Optimization Toolbox}} for {{Matlab Manual}},
  version 8.1..'' http://docs.mosek.com/9.0/toolbox/index.html, 2017.

\bibitem{Reynolds2020}
T.~P. Reynolds and M.~Mesbahi, ``The {{Crawling Phenomenon}} in {{Sequential
  Convex Programming}},'' in {\em American {{Control Conference}} ({{ACC}})},
  pp.~3613--3618, {IEEE}, July 2020.

\end{thebibliography}


\begin{thebibliography}{10}

\bibitem{Mao2019a}
Y.~Mao, M.~Szmuk, X.~Xu, and B.~Acikmese, ``Successive {{Convexification}}: {{A
  Superlinearly Convergent Algorithm}} for {{Non-convex Optimal Control
  Problems}},'' {\em arXiv preprint}, 2019.

\bibitem{Mao2016}
Y.~Mao, M.~Szmuk, and B.~Acikmese, ``Successive convexification of non-convex
  optimal control problems and its convergence properties,'' in {\em
  {{Conf.}} on {{Decis.}} and {{Control}}}, pp.~3636--3641, 
  {IEEE}, 2016.

\bibitem{Acikmese2007}
B.~A{\c c}{\i}kme{\c s}e and S.~R. Ploen, ``Convex {{Programming Approach}} to
  {{Powered Descent Guidance}} for {{Mars Landing}},'' {\em J. of Guid.,
  Control and Dyn.}, vol.~30, no.~5, p.~1353--1366, 2007.

\bibitem{Akmee2011a}
B.~A{\c c}{\i}kme{\c s}e and L.~Blackmore, ``Lossless convexification of a
  class of optimal control problems with non-convex control constraints,'' {\em
  Automatica}, vol.~47, no.~2, pp.~341--347, 2011.

\bibitem{Harris2021}
M.~W. Harris, ``Optimal {{Control}} on {{Disconnected Sets Using Extreme Point
  Relaxations}} and {{Normality Approximations}},'' {\em IEEE Trans. on
  Autom. Control}, vol.~66, pp.~6063--6070, 2021.

\bibitem{Nocedal2006}
J.~Nocedal and S.~J. Wright, {\em Numerical {{Optimization}}}.
\newblock {Springer New York}, 2006.

\bibitem{Malyuta2022}
D.~Malyuta, T.~P. Reynolds, M.~Szmuk, T.~Lew, R.~Bonalli, M.~Pavone, and
  B.~A{\c c}{\i}kme{\c s}e, ``Convex {{Optimization}} for {{Trajectory
  Generation}}: {{A Tutorial}} on {{Generating Dynamically Feasible
  Trajectories Reliably}} and {{Efficiently}},'' {\em IEEE Control Syst.
  Mag.}, vol.~42, pp.~40--113, 2022.

\bibitem{Bonalli2022a}
R.~Bonalli, T.~Lew, and M.~Pavone, ``Analysis of {{Theoretical}} and
  {{Numerical Properties}} of {{Sequential Convex Programming}} for
  {{Continuous-Time Optimal Control}},'' {\em IEEE Trans. on Autom.
  Control}, vol.~68, pp.~4570--4585, 2022.

\bibitem{Bonalli2019b}
R.~Bonalli, A.~Cauligi, A.~Bylard, and M.~Pavone, ``{{GuSTO}}: {{Guaranteed
  Sequential Trajectory}} optimization via {{Sequential Convex Programming}},''
  in {\em Int. {{Conf.}} on {{Robot.}} and {{Automat.}}}, pp.~6741--6747, {IEEE}, 2019.

\bibitem{Boyd2004}
S.~Boyd and L.~Vandenberghe, {\em Convex {{Optimization}}}.
\newblock {Cambridge, England}: {Cambridge University Press}, 2004.

\bibitem{Yuille2003}
A.~L. Yuille and A.~Rangarajan, ``The {{Concave-Convex Procedure}}
  ({{CCCP}}),'' in {\em Adv. in {{Neural Inf. Process. Syst.}}},
  2001.

\bibitem{Boggs1989}
P.~T. Boggs and J.~W. Tolle, ``A {{Strategy}} for {{Global Convergence}} in a
  {{Sequential Quadratic Programming Algorithm}},'' {\em SIAM J. Numer. Anal.}, vol.~26, pp.~600--623,
  1989.

\bibitem{Gill2002}
P.~E. Gill, W.~Murray, and M.~A. Saunders, ``{{SNOPT}}: {{An SQP Algorithm}}
  for {{Large-Scale Constrained Optimization}},'' {\em SIAM J. on
  Optim.}, vol.~12, pp.~979--1006, 2002.

\bibitem{Mao2021}
Y.~Mao and B.~Acikmese, ``{{SCvx-fast}}: {{A Superlinearly Convergent
  Algorithm}} for {{A Class}} of {{Non-Convex Optimal Control Problems}},''
  {\em arXiv preprint}, 2021.

\bibitem{Szmuk2020}
M.~Szmuk, T.~P. Reynolds, and B.~A{\c c}{\i}kme{\c s}e, ``Successive
  {{Convexification}} for {{Real-Time Six-Degree-of-Freedom Powered Descent
  Guidance}} with {{State-Triggered Constraints}},'' {\em J. of Guid.,
  Control, and Dyn.}, vol.~43, pp.~1399--1413, 2020.

\bibitem{Bonalli2019a}
R.~Bonalli, A.~Bylard, A.~Cauligi, T.~Lew, and M.~Pavone, ``Trajectory
  {{Optimization}} on {{Manifolds}}: {{A Theoretically-Guaranteed Embedded
  Sequential Convex Programming Approach}},'' in {\em 
  {{Robot.}}: {{Sci.}} and {{Syst.}}}, 2019.

\bibitem{Bertsekas1982}
D.~P. Bertsekas, {\em Constrained {{Optimization}} and {{Lagrange Multiplier
  Methods}}}.
\newblock {Elsevier}, 1982.

\bibitem{Hestenes1969}
M.~R. Hestenes, ``Multiplier and gradient methods,'' {\em J. of
  Optim. Theory and Appl.}, vol.~4, pp.~303--320, 1969.

\bibitem{Powell1978}
M.~J.~D. Powell, ``Algorithms for nonlinear constraints that use lagrangian
  functions,'' {\em Math. Program.}, vol.~14, pp.~224--248, 1978.

\bibitem{Boyd2011}
S.~Boyd, N.~Parikh, E.~Chu, B.~Peleato, and J.~Eckstein, ``Distributed
  {{Optimization}} and {{Statistical Learning}} via the {{Alternating Direction
  Method}} of {{Multipliers}},'' {\em Found. and Trends
  in Mach. Learn.}, vol.~3, pp.~1--122, 2011.

\bibitem{Oguri2022f}
K.~Oguri and G.~Lantoine, ``Stochastic {{Sequential Convex Programming}} for
  {{Robust Low-thrust Trajectory Design}} under {{Uncertainty}},'' in {\em
  {{AAS}}/{{AIAA Astro. Special. Conf.}}}, 2022.

\bibitem{Bertsekas1976}
D.~P. Bertsekas, ``On {{Penalty}} and {{Multiplier Methods}} for {{Constrained
  Minimization}},'' {\em SIAM J. Control Optim.}, vol.~14, pp.~216--235,, 1976.

\if\shortOrFull1 
\bibitem{Oguri2023a}
K.~Oguri, ``Successive Convexification with Feasibility Guarantee via Augmented Lagrangian for Non-Convex Optimal Control Problems,''  {\em arXiv preprint}, 2023.
\fi

\bibitem{Grant2014}
M.~Grant and S.~Boyd, ``{{CVX}}: {{Matlab Software}} for {{Disciplined Convex
  Programming}}, version 2.1,'' 2014.

\bibitem{Mosek2017}
{Mosek ApS}, ``The {{MOSEK Optimization Toolbox}} for {{Matlab Manual}},
  version 8.1..'' http://docs.mosek.com/9.0/toolbox/index.html, 2017.

\bibitem{Reynolds2020}
T.~P. Reynolds and M.~Mesbahi, ``The {{Crawling Phenomenon}} in {{Sequential
  Convex Programming}},'' in {\em Amer. {{Control Conf.}}}, pp.~3613--3618,
  {IEEE}, 2020.

\end{thebibliography}

\bibliographystyle{ieeetr}

\end{document}